\documentclass[]{article}

\usepackage{amssymb}
\usepackage{amsthm}
\usepackage{amsmath}
\usepackage{anysize}
\usepackage{url}
\usepackage{color}
\usepackage{graphicx}
\usepackage{overpic}
\usepackage{subcaption}

\usepackage{paralist}
\usepackage{comment}
\usepackage[all]{nowidow}
\usepackage{xfrac}

\newcommand{\Hil}{\mathcal{H}}

\newcommand{\Acal}{\mathcal{A}}

\newcommand{\R}{\mathbb{R}}

\newcommand{\eps}{\varepsilon}
\newcommand{\ann}{B_{2a}(0)}
\newcommand{\dx}{\frac{\partial}{\partial x}}
\newcommand{\dy}{\frac{\partial}{\partial y}}

\makeatletter
\newcommand{\doublewidetilde}[1]{{%
  \mathpalette\double@widetilde{#1}%
}}
\newcommand{\double@widetilde}[2]{%
  \sbox\z@{$\m@th#1\widetilde{#2}$}%
  \ht\z@=.9\ht\z@
  \widetilde{\box\z@}%
}
\makeatother

\newcommand{\ph}{\varphi}

\newcommand{\vertiii}[1]{{\left\vert\kern-0.25ex\left\vert\kern-0.25ex\left\vert #1 
    \right\vert\kern-0.25ex\right\vert\kern-0.25ex\right\vert}}

\newtheorem{theorem}{Theorem}
\newtheorem{remark}[theorem]{Remark}

\newtheorem{proposition}[theorem]{Proposition}
\newtheorem{lemma}[theorem]{Lemma}
\newtheorem{corollary}[theorem]{Corollary}

\numberwithin{theorem}{section}
\numberwithin{equation}{section}

\DeclareMathOperator*{\argmin}{arg\,min}

\title{Gabor phase retrieval is severely ill-posed}
\author{Rima Alaifari and Philipp Grohs}

\begin{document}
\maketitle
\begin{abstract}
The problem of reconstructing a function from the magnitudes of its frame coefficients has recently been shown to be never uniformly stable in infinite-dimensional spaces \cite{cahill2016phase}. This result also holds for frames that are possibly continuous \cite{alaifari2017phase}. On the other hand, in finite-dimensional settings, unique solvability of the problem implies uniform stability.

A prominent example of such a phase retrieval problem is the recovery of a signal from the modulus of its Gabor transform. In this paper, we study Gabor phase retrieval and ask how the stability degrades on a natural family of finite-dimensional subspaces of the signal domain $L^2(\mathbb{R})$. We prove that the stability constant scales at least quadratically exponentially in the dimension of the subspaces. Our construction also shows that typical priors such as sparsity or smoothness promoting penalties do not constitute regularization terms for phase retrieval. 
\end{abstract}

\section{Introduction}\label{sec:intro}
Phase retrieval is a broad term for reconstruction problems, in which one aims at recovering a signal, a density distribution or some other quantity from complex- (or real-valued) measurements that do not contain phase (or sign) information. The most prominent example is X-ray crystallography that can be modeled as the reconstruction of an electron density distribution from the modulus of its Fourier transform. 

More formally, one can state the phase retrieval problem for any Hilbert space $\Hil$ and an associated measurement system $(\psi_\lambda)_{\lambda \in \Lambda} \subset \Hil$ which constitutes a frame (that is not necessarily discrete, depending on the index set $\Lambda$). While any $f \in \Hil$ can be uniquely and stably recovered from $\{ \langle f, \psi_\lambda \rangle_\Hil \}_{\lambda \in \Lambda}$, these properties are not so evident for the reconstruction of $f$ from the phaseless measurements $\{ |\langle f, \psi_\lambda \rangle_\Hil| \}_{\lambda \in \Lambda}$.

In the real-valued setting, i.e. when $f$ and the measurement system $(\psi_\lambda)_{\lambda \in \Lambda}$ are real-valued, the so-called \emph{complement property (CP)} is a necessary and sufficient condition on $(\psi_\lambda)_{\lambda \in \Lambda}$ for injectivity of the phase retrieval problem, both in the case when $\Hil$ is finite-dimensional \cite{balan2006signal} and infinite-dimensional \cite{cahill2016phase, alaifari2017phase}. In the complex-valued case, the CP is only a necessary condition for injectivity \cite{balan2006signal, cahill2016phase, alaifari2017phase}. A stronger condition on the family $(\psi_\lambda)_{\lambda \in \Lambda}$ is the \emph{strong complement property (SCP)}. Introduced in \cite{bandeira2014saving} for the finite-dimensional case, it is shown that the SCP is a necessary and sufficient condition for stable phase retrieval in the real-valued setting and also conjectured, that the SCP is a necessary condition for stability in the complex-valued setting. In \cite{alaifari2017phase}, we consider the phase retrieval problem for the general framework of infinite-dimensional Banach spaces and associated continuous Banach frames. We prove that in this generalized setting the SCP is necessary and sufficient for stable phase retrieval in the real-valued case and indeed necessary in the complex-valued case. In addition, we show that the SCP can never hold in infinite-dimensional spaces, hence phase retrieval is never uniformly stable in infinite dimensions. Such a result has been derived before in \cite{cahill2016phase} for the case of discrete frames of Hilbert spaces without employing the SCP. In the case of a measurement system consisting of the Cauchy wavelet transform, uniqueness and the lack of uniform stability have first been shown in \cite{mallat2015phase}. 

While these results suggest that the instability of phase retrieval is truly inherent of the problem, numerical experiments indicate that phase retrieval is unstable whenever the phaseless measurements $\{ |\langle f, \psi_\lambda \rangle_\Hil| \}_{\lambda \in \Lambda}$ carry modes of silence, i.e., whenever the magnitude measurements are concentrated on at least two disconnected components of $\Lambda$ and small outside of these regions. This observation has led the authors, together with Ingrid Daubechies and Rujie Yin, to study the phase retrieval problem in a relaxed notion of \emph{atoll functions} for which stability can in fact be restored \cite{alaifari2016stable} when the measurement system is the Gabor transform or the Cauchy wavelet transform. One of the authors, together with Martin Rathmair, has recently derived significantly improved stability estimates for the Gabor transform case and, in particular, showed that all instabilities are of this kind \cite{grohs2017stable}, thereby providing a rather complete picture of the stability properties of Gabor phase retrieval.

This paper likewise studies the reconstruction of functions in $L^2(\mathbb{R})$ from the magnitude of their Gabor transforms, i.e., of their windowed Fourier transforms with Gaussian window $\ph(t)=e^{-\pi t^2}$. This problem arises in ptychographic imaging, a variation of X-ray crystallography in which an electron density distribution is reconstructed from the modulus of its short-time/windowed Fourier transform, or STFT in short. Gabor phase retrieval can also be employed in certain audio processing applications such as the phase vocoder. A phase vocoder is a device that modifies audio signals (such as time-stretching or pitch-shifting them) and can be implemented by retrieving phase information from the STFT modulus.

The Gabor transform $V_\ph f$ of a function $f \in L^2(\mathbb{R})$ is defined as
\begin{equation}
V_\ph f (x,y) := \int_{\mathbb{R}} f(t) \ph(t-x) e^{-2 \pi i t y}  dt.
\end{equation}
The problem of Gabor phase retrieval then amounts to recovering $f$ from $\{|V_\ph f(x,y)|\}_{(x,y) \in \mathbb{R}^2}$. Note that phase retrieval is always to be understood up to a global phase (or sign) factor: Given magnitude measurements of $f$, one will never be able to distinguish $f$ from $\tau f$, for any $\tau \in S^1:=\{z \in \mathbb{C} \mid |z|=1\}$. Thus, to make the notions of uniqueness and stability meaningful, one has to discard global phase factors which can be done by introducing the following metric:
$$
	\mbox{dist} (f,g) := \inf_{\tau \in S^1} \| f - \tau g \|_{L^2(\mathbb{R})}.
$$ 
Gabor phase retrieval would be stable (in a strong sense), if there existed uniform constants $c_1, c_2>0$ such that for all $f, g \in L^2(\mathbb{R})$
\begin{equation}\label{eqn:stab}
c_1 \mbox{dist} (f,g) \leq \| |V_\ph f| - |V_\ph g| \|_{L^2(\mathbb{R}^2)} \leq c_2 \mbox{dist} (f,g).
\end{equation}
As discussed above, we know that such uniform stability cannot hold. Instead, we want to quantify this instability by studying the degradation of stability on nested finite-dimensional subspaces of $L^2(\mathbb{R})$. We note that in \cite{cahill2016phase}, Cahill, Casazza \& Daubechies give such an example for a sign retrieval problem in 1D: the recovery of a real-valued bandlimited function from its unsigned samples. There, the authors consider finite-dimensional subspaces of the Paley-Wiener space and demonstrate that the stability constant for this sign retrieval problem grows exponentially in the dimension of the subspace.

\textbf{Contributions.} In this paper, we study two questions for Gabor phase retrieval:
\begin{itemize}
\item Can the degree of ill-posedness of this inverse problem be quantified in some sense? 
\item Is it possible to regularize the problem with a prior on the smoothness or sparsity of the solution?
\end{itemize}
For answering the first question, we construct pairs of functions $(f_a^+, f_a^-)$ that depend on a positive parameter $a \in \mathbb{R}$. Then, for (\ref{eqn:stab}) to hold for some $c_1$ and this particular pair of functions, we show $c_1 \lesssim e^{-a^2 \pi/2}$. Such exponential degradation of stability for inverse problems is often referred to as \emph{severe} ill-posedness, a term more commonly used for problems with linear operators that have purely discrete spectrum (cf. Remark \ref{rem:main}). Our construction also gives an insight to the answer of the second question: the functions used are both smooth and have sparse approximations in any time-frequency/time-scale representation. Therefore, this ill-posed problem cannot be regularized through a smoothness penalty or a sparsity-promoting penalty.

\textbf{Outline.} In Section \ref{sec:inverse-problem} we briefly give the link between the phase retrieval operator and general inverse problems theory for non-linear operators, showing that the phase retrieval operator is not compact. In Section \ref{sec:lowerbound}, we construct a pair of functions, depending on some parameter $a$, and state that for such a pair the stability constant of Gabor phase retrieval is at least exponential in $a^2$. The parameter $a$ can also be linked to the dimension of a subspace of $L^2(\mathbb{R})$ so that the stability constant scales at least quadratically exponentially in the dimension of the problem. Based on our construction, one can then deduce that the typical remedy to ill-posedness -- adding a regularization penalty -- does not help in phase retrieval, see Section \ref{subsec:regularization}. The proof of our main proposition is presented in Section \ref{sec:proofs} and some proofs of Section \ref{sec:sparse_reg} are outsourced to Section \ref{sec:proofs_sparse_reg}.

\section{Inversion of the Gabor phase retrieval operator}\label{sec:inverse-problem}

We may formulate Gabor phase retrieval as the inverse problem with the nonlinear forward operator

\begin{equation}\label{def:fwdop}
\Acal_\ph: L^2(\mathbb{R})/S^1 \to L^2(\mathbb{R}^2,\mathbb{R}_0^+),\quad
	f\mapsto |V_\ph f|,
\end{equation}
where $|V_\ph f|$ denotes the Gabor transform magnitude, $|V_\ph f|(x,y):= |V_\ph f(x,y)|$, for all $(x,y) \in \mathbb{R}^2$. The operator $\Acal_\ph$ has the following properties:
\begin{proposition}\label{prop:weak-seq-closed}
The operator $\Acal_\ph$ is injective and weakly sequentially closed. Moreover, its range $\mathcal{R}(\Acal_\ph)$ is closed and $\Acal_\ph^{-1}$ is continuous on $\mathcal{R}(\Acal_\ph)$.
\end{proposition}
\begin{proof}
The injectivity of $\Acal_\ph$ is well-known, see e.g. \cite{grohs2017stable}. The fact that $\mathcal{R}(\Acal_\ph)$ is closed and the continuity of $\Acal_\ph^{-1}$ on $\mathcal{R}(\Acal_\ph)$ have been shown in  \cite{alaifari2017phase}. It only remains to establish that the operator is weakly sequentially closed. For this, we need to show that for a sequence $\{f_n\}_{n \in \mathbb{N}} \subset L^2(\mathbb{R})$ converging weakly to $f \in L^2(\mathbb{R})$ the weak convergence of $\{|V_\ph f_n|\}_{n \in \mathbb{N}}$ to $v \in L^2(\mathbb{R}^2,\mathbb{R}_0^+)$ implies $|V_\ph f| = v$. This, however, is immediate, as the weak convergence of $\{f_n\}_{n \in \mathbb{N}}$ implies pointwise convergence of $|V_\ph f_n|$ to $|V_\ph f|$. In $L^2(\mathbb{R}^2)$, the pointwise and weak limit of a sequence must coincide, so that $|V_\ph f| = v$.
\end{proof}
For compact linear operators $C:X \to Y$, it is well known that a sufficient condition for ill-posedness is injectivity of $C$, provided that $X$ is infinite-dimensional \cite{engl1996regularization}. A similar result exists for nonlinear operators $A: X \to Y$. Proposition 10.1 in \cite{engl1996regularization} implies that if $A$ is compact, injective and weakly sequentially closed and $X$ is an infinite-dimensional separable Hilbert space, then $A^{-1}$ is not continuous (the original statement is stronger in that it implies ill-posedness from local properties of the operator). 

This result seems in contrast to Proposition \ref{prop:weak-seq-closed}. The reason these results are not in contradiction is that $\Acal_\ph$ is not compact:
\begin{proposition}\label{prop:not-compact}
The operator $\Acal_\ph$ is not compact.
\end{proposition}
To prove this result we use the following 
\begin{lemma} Let $L>0$ and
$$
B:= \{f \in L^2(\mathbb{R})/S^1: \ \|f\|_{L^2(\mathbb{R})} \leq L\}.
$$
Then, the set 
$$
\mathcal{V}_B:= \{ V_\ph f: f \in B\} \subset L^2(\mathbb{R}^2)
$$
is strongly equicontinuous, meaning that for all $\eps >0$ and all $(x,y) \in \mathbb{R}^2$, there exists an open set $U_\eps((x,y)) \subset \mathbb{R}^2$ such that for all $(x',y') \in U_\eps ((x,y))$,
$$
\sup_{f \in B} |V_\ph f (x,y) - V_\ph f(x',y')| \leq \eps.
$$
\end{lemma}
\begin{proof}
Denoting by $\ph_{x,y}$ the modulated and time-shifted Gaussian $\ph_{x,y}(t):= e^{2 \pi i t y} e^{-\pi(t-x)^2}$, the boundedness of elements in $B$ implies
$$
 |V_\ph f (x,y) - V_\ph f(x',y')| \leq L \cdot \| \ph_{x,y} - \ph_{x',y'} \|_{L^2(\mathbb{R}^2)}.
$$
The statement then follows, since for every $\eps > 0$, there is an open set $U_\eps ((x,y)) \subset \mathbb{R}^2$ such that
$$
\| \ph_{x,y} - \ph_{x',y'}\|_{L^2(\mathbb{R}^2)} \leq \frac{\eps}{L}.
$$
\end{proof}
With this property at hand, we can now give the proof of the non-compactness of $\Acal_\ph$:
\begin{proof}[Proof of Proposition \ref{prop:not-compact}]
Let $L>0$ and consider the image $\Acal_\ph(B)$ of the set 
$$
B:= \{f \in L^2(\mathbb{R})/S^1: \ \|f\|_{L^2(\mathbb{R})} \leq L\}.
$$
Since $\mathcal{V}_B$ is strongly equicontinuous by our previous lemma, we can use Corollary A.4 in our previous work \cite{alaifari2016stable}, to state that $\Acal_\ph(B)$ is relatively compact if and only if $\mathcal{V}_B$ is relatively compact. Theorem A.3 in \cite{alaifari2016stable} gives a characterization of $\mathcal{V}_B$ being relatively compact: the property holds if and only if $\mathcal{V}_B$ is bounded and for all $\eps>0$ there is a compact set $C_\eps \subset \mathbb{R}^2$ such that
$$
\sup_{w \in \mathcal{V}_B} \|w \|_{L^2(\mathbb{R}^2 \backslash C_\eps)} \leq \eps.
$$ 
It is clear that this cannot hold: for any $\eps >0$ and any choice of $C_\eps$ one can simply find a signal in $B$, for which the Gabor transform is concentrated outside of $C_\eps$ (realized by a sufficiently large time-delay or frequency shift). Hence, $\mathcal{V}_B$ and consequently also $\Acal_\ph(B)$ are not relatively compact.
\end{proof}
The continuity of $\Acal_\ph^{-1}$ is of course not as significant as in the case of linear operators: as we have shown in \cite{alaifari2016stable}, $\Acal_\ph^{-1}$ is not uniformly continuous, so that a quantitative result as in (\ref{eqn:stab}) cannot hold. In fact, although for $\Acal_\ph (f_n)$ converging to $u$, $f_n$ is guaranteed to converge to $\Acal_\ph^{-1}(u)$, so that the problem is \emph{not} ill-posed in the classical sense, the convergence can be arbitrarily slow.

In what follows, we will use the term \emph{ill-posed} to include also this lack of strong stability. We intend to characterize this ill-posedness with our goal being two-fold: on the one hand, we show that the degree of this instability is in some sense \emph{severe}. On the other hand, we demonstrate the stability degradation for very simple pairs of signals: they are smooth and also sparse in any time-frequency/time-scale localized representation. Therefore, classical regularization penalties become infeasible for phase retrieval.

\section{Analysis of the phase retrieval instability}\label{sec:lowerbound}

\subsection{A simple parameter-dependent couple}\label{subsec:couple}

First, let us introduce the translation operator $T_a$ on $L^2(\mathbb{R})$ defined as
\begin{equation}\label{eqn:translation}
T_a f(t):=f(t-a)
\end{equation}
for $f \in L^2(\mathbb{R})$ and $a \in \mathbb{R}$.

The goal is to construct two functions $f_a^+$, $f_a^- \in L^2(\mathbb{R})$, for which the Gabor transform measurements are close to each other in absolute value but such that $\|f_a^+-\tau f_a^-\|_{L^2(\R)}$ is not small for any phase factor $\tau \in S^1$. We proceed by defining the shifted Gaussian functions
\begin{equation}\label{eqn:gauss_shift}
	u_a := T_a \ph, \quad u_{-a}:= T_{-a} \ph,
\end{equation}
and fixing our pair of functions to be
\begin{equation}\label{eqn:f_a}
	f_a^+:= u_{-a} + u_a, \qquad f_a^-:= u_{-a}-u_a.
\end{equation}
For $a>0$ not too small, the Gabor transforms $V_\ph f_a^+$ and $V_\ph f_a^-$ are concentrated on two disconnected components. We note that 
$$
	f_a^+ - f_a^- = 2 u_a, \qquad f_a^+ + f_a^- = 2 u_{-a},
$$
and hence,
\begin{equation}\label{eqn:lower-bound}
	\inf_{\tau \in S^1} \| f_a^+ - \tau f_a^-\|_{L^2(\R)} = \min_{\tau \in \{ \pm 1\}} \| f_a^+ - \tau f_a^-\|_{L^2(\R)}= 2 \| \ph \|_{L^2(\R)} = 2^{3/4},
\end{equation}
where the first equality is due to $f_a^+$ and $f_a^-$ being real-valued. 

We remark that if we confine the parameter $a$ to a discrete set $\{a_k:= k \, q \mid k \in \mathbb{N}\}$, for some fixed $q>0$, then the spaces
$$
	\Hil_{k} := \mbox{span } \{ T_s \ph \mid s \in \{-k \, q, -(k-1) \, q, \dots, (k-1) \, q, k \, q \}\}
$$
are nested finite-dimensional subspaces of $L^2(\mathbb{R})$ and $f_{a_k}^+, f_{a_k}^- \in \Hil_k$. Hence, bounding 
\begin{equation}\label{eqn:fin-meas}
\| |V_\ph f_{a_k}^+|-|V_\ph f_{a_k}^-| \|_{L^2(\mathbb{R}^2)}
\end{equation}
from above yields a lower bound on the stability constant of the phase retrieval problem in the finite-dimensional subspace $\Hil_k$. As we show in the next section, the term in (\ref{eqn:fin-meas}) is bounded from above by an expression decaying exponentially in $a_k^2$. Therefore, the deterioration of stability of phase retrieval in $\Hil_k$ is such that the stability constant grows exponentially in $a_k^2$. 

In fact, we show a stronger result by upper bounding 
$$
	\| |V_\ph f_{a_k}^+|-|V_\ph f_{a_k}^-| \|_{W^{1,2}(\mathbb{R}^2)},
$$
where the Sobolev norm $\| \cdot \|_{W^{1,2}(\mathbb{R}^2)}$ of $F \in W^{1,2}(\mathbb{R}^2)$ is defined as 
$$
	\| F \|_{W^{1,2}(\mathbb{R}^2)} := \| F \|_{L^2(\mathbb{R}^2)} + \| \nabla F  \|_{L^2(\mathbb{R}^2)}.
$$
\begin{figure}[h!]
\centering
	\includegraphics[width = 0.5\textwidth]{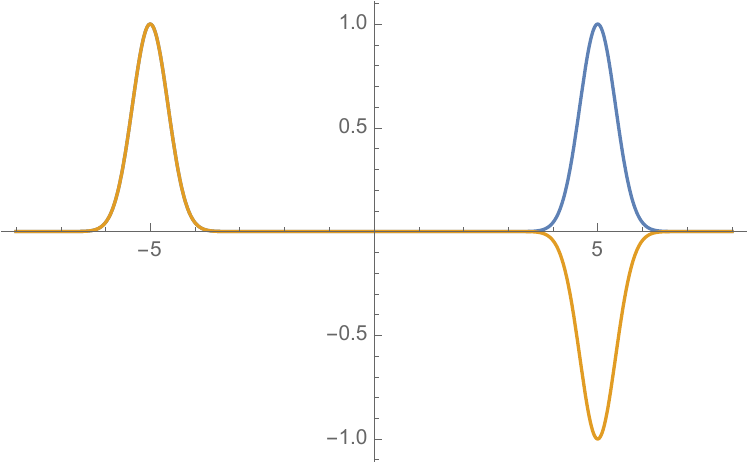}
	\caption{The functions $f_a^+$ (blue) and $f_a^-$ (orange) for $a=5$.}
\end{figure}
\begin{figure}[h!]
	\centering
	\begin{subfigure}[t]{.4\textwidth}
	\includegraphics[width = \textwidth]{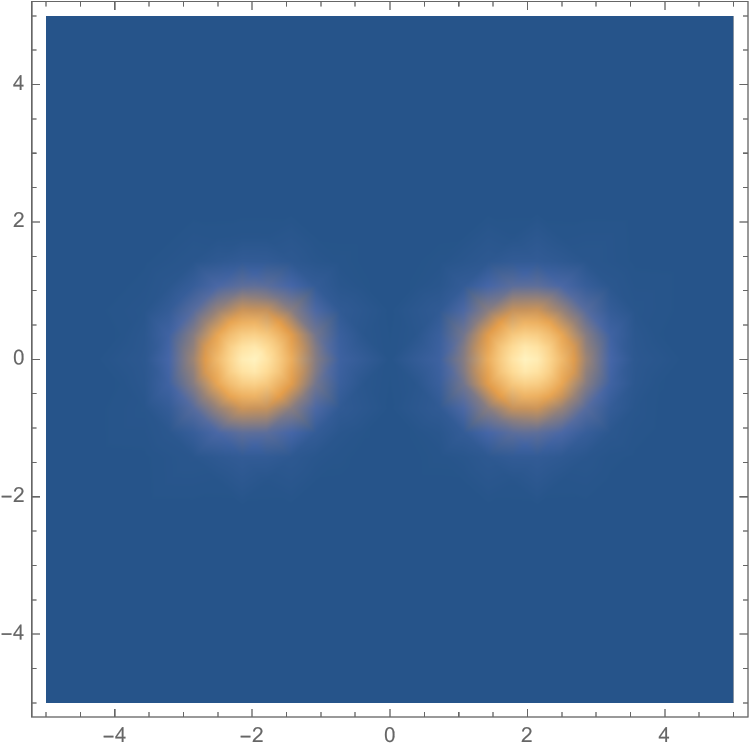}
	\caption{The modulus of the Gabor transform of $f_a^+$ for $a=2$.} \label{fig:V_f_plus}
	\end{subfigure}
	\hfill
	\begin{subfigure}[t]{.4\textwidth}
	\includegraphics[width = \textwidth]{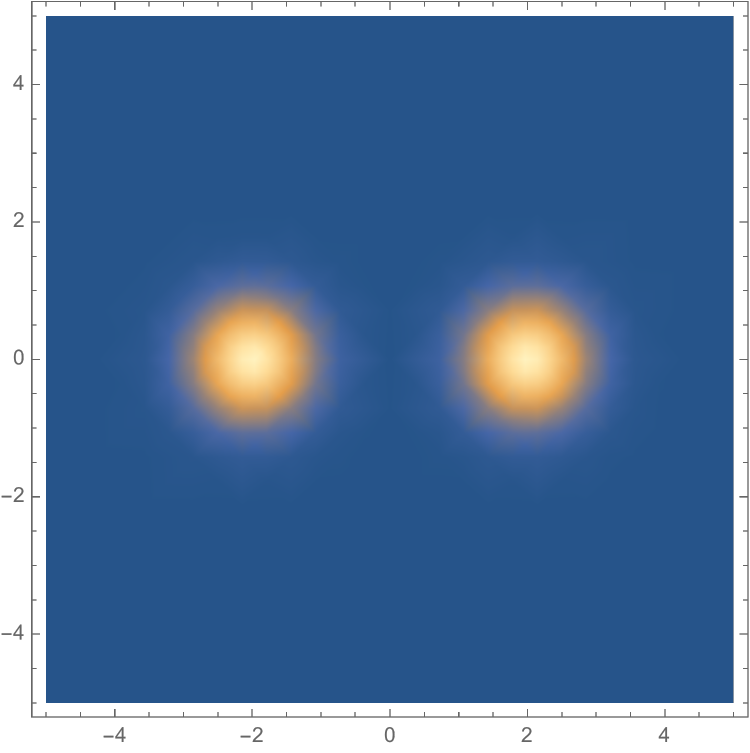}
	\caption{The modulus of the Gabor transform of $f_a^-$ for $a=2$.}\label{fig:V_f_minus}
	\end{subfigure}\\[1em]
	\begin{subfigure}[t]{.45\textwidth}
	\includegraphics[width = \textwidth]{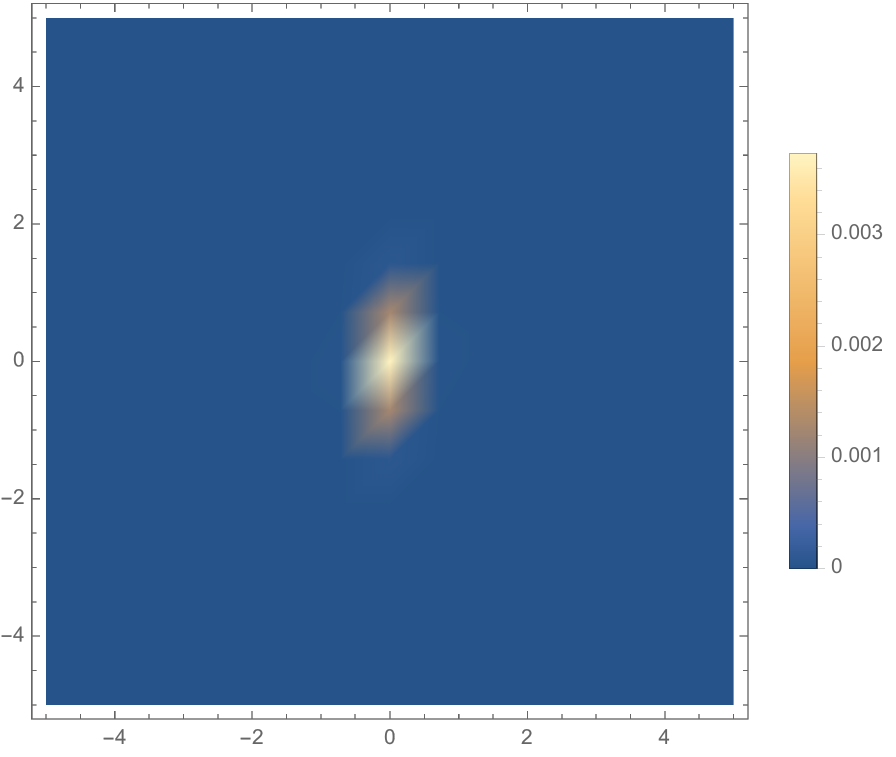}
	\caption{The difference between $|V_\ph f_2^+|$ and $|V_\ph f_2^-|$.}\label{fig:f_diff}
	\end{subfigure}
	\label{fig:f_a}
\end{figure}

\subsection{Severe ill-posedness}
In terms of stable reconstruction from the measurements we obtain that the $W^{1,2}$-distance between the measurements decays exponentially in $a^2$:
\begin{proposition}\label{prop:upper-bound}
There exists a uniform constant $C>0$ such that for all $a>0$ and for all $k \in (0, \pi/2)$,
\begin{equation}\label{eqn:upper-bound-gabor}
		\|  |V_\ph f_a^+|- |V_\ph f_a^-|  \|_{W^{1,2}(\mathbb{R}^2)}  
		\leq C \cdot e^{-k \cdot a^2}.
\end{equation}
\end{proposition}
\begin{proof}
Combining Lemma \ref{lem:normbd}, Lemma \ref{lem:normbd-dx} and Lemma \ref{lem:normbd-dy} yields the result.
\end{proof}

Together with (\ref{eqn:lower-bound}), this implies that for a stability estimate to hold, the stability constant scales at least exponentially in $a^2$:
\begin{corollary}\label{cor:stab-gabor-ex}
There exists a uniform constant $\widetilde{C}>0$ such that for all $a>0$ and for all $k \in (0, \pi/2)$,
\begin{equation}\label{eqn:lbd1}
\min_{\tau \in \{\pm1\}} \| f_a^+ - \tau f_a^-\|_{L^2(\R)} \geq \widetilde{C} e^{k \cdot a^2} \|  |V_\ph f_a^+|- |V_\ph f_a^-|  \|_{W^{1,2}(\mathbb{R}^2)}.
\end{equation}
\end{corollary}
\begin{proof}
Equation (\ref{eqn:lower-bound}) implies
$$
\inf_{\tau \in S^1} \| f_a^+ - \tau f_a^-\|_{L^2(\R)} \geq 1.
$$
Applying Proposition \ref{prop:upper-bound} yields (\ref{eqn:lbd1}). 
\end{proof}

\begin{remark}\label{rem:main}
In inverse problems, the ill-posedness of a linear operator equation can be quantified if the operator (or its normal form) has a discrete spectrum: instabilities are caused by accumulation of the eigen- (or singular) values of the operator at zero. Depending on the decay rate one typically speaks of \emph{mild} (polynomial rate) or \emph{severe} ill-posedness (exponential rate). In a similar way, the above result quantifies that the non-linear inverse problem of Gabor phase retrieval is also \emph{severely} ill-posed: the stability degrades at a rate that is at least exponential in the parameter $a$ that describes the moving apart of two Gaussians. 
\end{remark}

\begin{remark}
As mentioned in Section \ref{sec:intro}, a stability estimate for phase retrieval can be shown for the infinite-dimensional setting, if the notion of phase reconstruction is relaxed. Instead of recovering up to one global phase factor, so-called \emph{atoll domains} are introduced. These consist of $n$ bounded regions $D_j \subset \mathbb{C}$, $j=1,\dots,n$,  that are pairwise disjoint. One then seeks to reconstruct the Gabor transform up to a phase factor $\tau_j$ on each component $D_j$, allowing for different phase factors. We note that with this concept, our example presented in this section can be stably recovered. We refer to \cite{alaifari2016stable} for an introduction to atoll functions and stable phase retrieval and to \cite{grohs2017stable} for significantly improved stability constants.
\end{remark}

\subsection{Regularization penalties}\label{subsec:regularization}

Given the inherent nature of instability of the phase retrieval problem, the question of regularizing this inverse problem becomes inevitable. Regularization typically consists of adding a \emph{regularizing term} to the minimization functional, that amounts to a prior on the solution. Given some data $u$, instead of searching for the least-squares solution
$$
f_{\mbox{lsq}} := \argmin_{f \in L^2(\mathbb{R}) / S^1} \| \Acal_\ph (f) - u \|^2_{L^2(\mathbb{R}^2)}
$$
one would find the minimizer of 
\begin{equation}\label{eqn:reg}
\| \Acal_\ph (f) - u \|_{L^2(\mathbb{R}^2)}^2 + \tau \vertiii{f},
\end{equation}
where $\tau$ is some \emph{regularization parameter} and the term $ \vertiii{f}$ is a \emph{regularization penalty} that carries a-priori information that one might have on the solution. For example, if one knows that the solution has its $L^2-$norm bounded by $b>0$, $\|f\|_{L^2(\mathbb{R})}\leq b$, then, a solution with small $L^2-$norm can be enforced by adding the penalty $\|f\|_{L^2(\mathbb{R})}^2$.

Typical choices for regularization penalties on functions $f \in L^2(\mathbb{R})$ include terms $\vertiii{\cdot}$ that favour smooth functions or functions with small bounded variation. It is clear, however, that such regularization penalties would not improve the stability properties of the phase retrieval problem. The example constructed in Section \ref{subsec:couple} consists of functions $f_a^+$, $f_a^-$, that are sums of two translated Gaussians, so that they enjoy infinite regularity and have small variation. A regularization penalty favouring such properties would not stably discriminate between $f_a^+$ and $f_a^-$.

\subsubsection{Sparsity promoting penalties}\label{sec:sparse_reg}

Another family of regularization penalties consists of so-called \emph{sparsity promoting} terms. The underlying idea is that the a priori knowledge on the solution $f \in L^2(\mathbb{R})$ can be expressed through its representation in a \emph{frame} of $L^2(\mathbb{R})$. For a Hilbert space $\mathcal{H}$, a frame is a family $\{ \psi_\lambda \}_{\lambda \in \Lambda} \subset \mathcal{H}$, where the index set $\Lambda$ is discrete and for which there exist bounds $A, B>0$, such that 
\begin{equation}\label{eqn:frame}
A \|f\|_{\mathcal{H}}^2 \leq \sum_{\lambda \in \Lambda} |\langle f, \psi_\lambda \rangle|^2 \leq B \|f\|_{\mathcal{H}}^2 \quad \mbox{ for all } f \in \mathcal{H}.
\end{equation}
Given a frame $\{\psi_\lambda\}_{\lambda \in \Lambda}$ of $L^2(\mathbb{R})$, one could choose a regularization term for (\ref{eqn:reg}), that is a weighted $\ell^p-$norm of the frame coefficients, i.e. $\vertiii{f} = \| \{ \langle f, \psi_\lambda \rangle \}_{\lambda \in \Lambda} \|_{\ell^p_w}^p$. Choices of $p$ in $[1,2)$ would then favour functions $f$ with rapidly decaying frame coefficients \cite{daubechies2004iterative}. 

We consider the most prominent examples for time-frequency localized representations of signals: frames of wavelets and frames constituted by short-time Fourier transforms (STFTs). 

Given a \emph{mother wavelet} $\psi \in L^2(\mathbb{R}),$ a \emph{scaling function} $\chi \in L^2(\mathbb{R})$ and  parameters $\alpha>1, \beta>0$, a wavelet system can be introduced as
\begin{equation}\label{eqn:wavelet-sys}
\{ \chi_{0,k}:= \chi(\cdot-\beta k) \}_{k \in \mathbb{Z}} \cup \{ \psi_{j,k}(t) := \alpha^{j/2} \psi(\alpha^j t - \beta k)\}_{j \in \mathbb{N}_0, k \in \mathbb{Z}}.
\end{equation}
In this case, the (wavelet) frame condition (\ref{eqn:frame}) becomes
$$
A \|f\|_{L^2(\mathbb{R})}^2 \leq \sum_{k \in \mathbb{Z}} |\langle f, \chi_{0,k} \rangle|^2 + \sum_{j \in \mathbb{N}_0, k \in \mathbb{Z}} |\langle f, \psi_{j,k} \rangle|^2\leq B \|f\|_{L^2(\mathbb{R})}^2 \quad \mbox{ for all } f \in L^2(\mathbb{R}). 
$$
For wavelet frames, the sparsity promoting regularization penalty is typically chosen as the following weighted $\ell^p-$norm of the coefficients:
\begin{equation}\label{eqn:wavelet-penalty}
\vertiii{f} := \| f \|_{s,p}^p := \sum_{k \in \mathbb{Z}} |\langle f, \chi_{0,k} \rangle|^p + \sum_{j \in \mathbb{N}_0} \alpha^{j \sigma p} \sum_{k \in \mathbb{Z}} | \langle f, \psi_{j,k} \rangle|^p,
\end{equation}
where $\sigma = s+1/2-1/p$ and $p \in [1,2)$. If $\chi, \psi \in C^M(\mathbb{R})$ for $M>s$, then $\| \cdot \|_{s,p}$ is a norm equivalent to the \emph{Besov space} norm $\| \cdot \|_{B_{p,p}^s}$. The parameter $s$ indicates the smoothness of the function class $B_{p,p}^s$. Roughly speaking, the Besov space $B_{p,p}^s$ consists of functions that have $s$ derivatives in $L^p$.

Similarly, we will consider frames of short-time Fourier transforms. Given a window $g$ and parameters $x_0, y_0>0$, the system defined as $\{g_{n,k}(t) := e^{2 \pi i k y_0 t} g(t-n x_0)\}_{n,k \in \mathbb{Z}}$ is an STFT frame of $L^2(\mathbb{R})$ if there exist constants $A,B>0$ such that
$$
A \|f\|_{L^2(\mathbb{R})}^2 \leq \sum_{n,k \in \mathbb{Z}} |\langle f, g_{n,k} \rangle|^2 \leq B  \|f\|_{L^2(\mathbb{R})}^2 \quad \mbox{ for all } f \in L^2(\mathbb{R}).
$$
A sparsity promoting penalty in this setting would be of the form
\begin{equation}\label{eqn:stft-penalty}
\vertiii{f} := \| \{ \langle f, g_{n,k} \rangle \} \|_{\ell_w^{p}}^p:= \sum_{k,n \in \mathbb{Z}} |\langle f, g_{n,k} \rangle |^p w(n x_0, k y_0)^p,
\end{equation}
where for some $s \geq 0$, the weight is defined as $w(x,y):=\big((1+|x|)(1+ |y|)\big)^s$ or an equivalent weight of polynomial type.

We now want to show that 
\begin{enumerate}
\item  \label{item1} the functions $f_a^+$, $f_a^-$ can be sparsely approximated in wavelet frames or STFT frames, i.e. $$\vertiii{f_a^+}, \vertiii{f_a^-} < \infty, \mbox{ and}$$
\item $\Big|\vertiii{f_a^+}- \vertiii{f_a^-}\Big| \leq C_L a^{-L}$ for some $L \geq 1$ and a constant $C_L.$
\end{enumerate}
The second property is important to show that a regularization scheme as in (\ref{eqn:reg}) can not stably distinguish $f_a^+$ from $f_a^-$ for any of the considered sparsity promoting penalties. For item \ref{item1}, we note that since $f_a^+$, $f_a^-$ are linear combinations of two translated Gaussians, we only need to state that $\vertiii{\ph} < \infty$ for $\ph(t) = e^{- \pi t^2}$ and the norms defined as in (\ref{eqn:wavelet-penalty}) and (\ref{eqn:stft-penalty}). This is of course not surprising (under mild assumptions on the window or on the mother wavelet and scaling function), due to the nice decay properties of $\ph$ in both time and frequency.

More precisely, for the case of an expansion in an STFT frame with a window $g$ in the Schwartz space $\mathcal{S}(\mathbb{R})\backslash \{0\}$, we have the following
\begin{proposition}\label{prop:finite_lp_stft}
Let $g \in \mathcal{S}(\mathbb{R}) \backslash \{0\}$, $s \geq 0$ and $1 \leq p \leq \infty$. Then,
$$
 \| \{ \langle \ph, g_{n,k} \rangle \} \|_{\ell_w^{p}} < \infty.
$$
\end{proposition}
\begin{proof}
Since $\ph, g \in \mathcal{S}(\mathbb{R})$, $V_g \ph \in \mathcal{S}(\mathbb{R}^2)$ \cite[Theorem 11.2.5]{grochenig}. Thus, 
$$\|V_g \ph\|_{L_w^{p,p}} := \Big(\iint |V_g \ph(x,y)|^p w(x,y)^p dxdy\Big)^{1/p} < \infty.$$ On the other hand, classical results \cite{grochenig} guarantee
$$
\| \{ \langle \ph, g_{n,k} \rangle \} \|_{\ell_w^{p}} \leq C \cdot \|V_g \ph\|_{L_w^{p,p}}
$$
where the constant $C$ depends on $g, x_0, y_0$ and $s$.
\end{proof}

\begin{proposition}\label{prop:small_lp_stft}
Let $g \in \mathcal{S}(\mathbb{R}) \backslash \{0\}$ and $\vertiii{\cdot}:= \| \{ \langle \cdot, g_{n,k} \rangle \}\|_{\ell_w^p}^p$ with $s \geq 0$ and $p \in [1,2]$. Then, for every $m \in \mathbb{N}$, $m > sp+1$, there exists a constant $C_{m,x_0,y_0}>0$ that depends on only $m, x_0$ and $y_0$, s.t.
$$
\Big|\vertiii{f_a^+}- \vertiii{f_a^-}\Big| \leq C_{m,x_0,y_0} \big(1 + a\big)^{s p - m +1}.
$$
\end{proposition}
\begin{proof}
See Section \ref{sec:proofs_sparse_reg}.
\end{proof}
For the case of wavelet frames we employ standard results (see e.g. \cite{mallat1999wavelet}) and adapt them for our case to obtain specific decay rates for the wavelet coefficients of $\ph$. If $\psi$ has $m \in \mathbb{N}$ vanishing moments, i.e.
$$
\int_{\mathbb{R}} x^\ell \psi(x) dx = 0 \quad \mbox{ for all } \ell \in \{0, \dots, m-1 \},
$$
 and both $\psi$ and $\chi$ have sufficient spatial decay, then the wavelet and scaling coefficients of $\ph(t)=e^{-\pi t^2}$ decay at the order $(\alpha^j \beta |k|)^{-m-1}$:

\begin{proposition}\label{prop:wavelet-decay}
Let $\psi \in L^2(\mathbb{R})$ have $m \in \mathbb{N}$ vanishing moments and suppose that there exist constants $C_{2 m+2}$, $\widetilde{C}_{m+2}$ such that
\begin{align}
|\psi(x)| &\leq \frac{C_{2m+2}}{1+|x|^{2 m+2}} \mbox{ for all } x \in \mathbb{R},\label{eqn:wavelet-decay}\\
|\chi(x)| &\leq \frac{\widetilde{C}_{m+2}}{1+|x|^{m+2}} \mbox{ for all } x \in \mathbb{R}. \label{eqn:scaling-function-decay}
\end{align}
Then, there exists a constant $C>0$ depending on only $m$, such that
\begin{align}
|\langle \ph, \psi_{j,k} \rangle| &\leq C \, \alpha^{-j(m+3/2)} (\beta |k|)^{-m-1}, \label{eqn:wav-coeff-decay} \\
|\langle \ph, \chi_{0,k} \rangle| &\leq C \, (\beta |k|)^{-m-1}. \label{eqn:scaling-coeff-decay}
\end{align}
\end{proposition}

\begin{proof}
See Section \ref{sec:proofs_sparse_reg}.
\end{proof}
As an immediate consequence, the weighted $\ell^p-$norm (\ref{eqn:wavelet-penalty}) of the wavelet coefficients of $\ph$ is finite:
\begin{corollary}
For $\alpha>1, \beta>0$, let $\{ \chi_{0,k}, \psi_{j,k}\}_{j \in \mathbb{N}_0, k \in \mathbb{Z}}$ be a wavelet frame of $L^2(\mathbb{R})$, where $\psi$ has $m \in \mathbb{N}$ vanishing moments and $\psi$ and $\chi$ decay as in (\ref{eqn:wavelet-decay}) and (\ref{eqn:scaling-function-decay}), respectively. Then, for $p \geq 1$ and $s \leq m+1$,
\begin{equation}
\vertiii{ \ph} :=  \sum_{k \in \mathbb{Z}} | \langle \ph, \chi_{0,k} \rangle|^p + \sum_{j=0}^{\infty} \alpha^{j \sigma p} \sum_{k \in \mathbb{Z}} | \langle \ph, \psi_{j,k} \rangle |^p  < \infty,
\end{equation}
where again $\sigma = s+1/2-1/p$.
\end{corollary}

\begin{proposition}
As before, let $\{ \chi_{0,k}, \psi_{j,k}\}_{j \in \mathbb{N}_0, k \in \mathbb{Z}}$ be a wavelet frame of $L^2(\mathbb{R})$,  so that $\psi$ and $\chi$ decay as in (\ref{eqn:wavelet-decay}) and (\ref{eqn:scaling-function-decay}), respectively. If for $p \in [1,2]$ and $s \geq 0$, $\psi$ has $m \in \mathbb{N}$ vanishing moments with $2 m - \sigma p +3/2 > 0$, then
$$
\Big| \vertiii{f_a^+} - \vertiii{f_a^-} \Big| \leq C_{m,\alpha,\beta} \, a^{-m},
$$
for some constant $C_{m,\alpha,\beta}$ depending on only $m, \alpha$ and $\beta$.
\end{proposition}
\begin{proof}
The derivation can be done similarly to the one in Proposition \ref{prop:small_lp_stft}, employing the bounds (\ref{eqn:wav-coeff-decay}) and (\ref{eqn:scaling-coeff-decay}).
\end{proof}

\subsubsection{Penalties on the Gabor transform $V_\ph f$}

Algorithms for phase retrieval often recover the signal transform $V_\ph f$ instead of reconstructing $f$ directly. This way, the problem of Gabor phase retrieval can be formulated in two dimensions as reconstructing the phase of a function in $L^2(\mathbb{R}^2)$. Of course $f$ can be retained from $V_\ph f$, once the latter is known.

In this case, the minimizing functional becomes
$$
\argmin_{F \in L^2(\mathbb{R}^2) / S^1} \| |F| - u \|^2_{L^2(\mathbb{R}^2)}
$$
and one could add a regularization penalty on the Gabor transform:
$$
\argmin_{F \in L^2(\mathbb{R}^2) / S^1} \| |F| - u \|_{L^2(\mathbb{R}^2)}^2 + \tau \vertiii{F}
$$
for some regularization parameter $\tau$ and regularization term $\vertiii{\cdot}$ on $L^2(\mathbb{R}^2)$.

As in the previous section, we ask whether typical regularization penalties can stably discriminate the functions constructed in Section \ref{subsec:couple}, this time in the transform domain, i.e. $V_\ph f_a^+$ and $V_\ph f_a^-$. As one can see from Lemma \ref{lem:V_u}, these functions are the sums of two translated and modulated $2D$ Gaussian functions. Therefore it is clear that regularization penalties promoting smoothness will not improve the stability of the reconstruction problem.

Furthermore, since $V_\ph f_a^+$ and $V_\ph f_a^-$ are in the Schwartz space $\mathcal{S}(\mathbb{R}^2)$, their frame coefficients with respect to space-frequency or space-scale representations in $L^2(\mathbb{R}^2)$ are rapidly decaying. 

More precisely, for a window $\mathbf{g} \in \mathcal{S}(\mathbb{R}^2)$ and a signal $F \in \mathcal{S}(\mathbb{R}^2)$, the transform $V_\mathbf{g} F$ is an element of $\mathcal{S}(\mathbb{R}^4)$ \cite[Theorem 11.2.5]{grochenig}. From this one can deduce analogous results  to Propositions \ref{prop:finite_lp_stft} and \ref{prop:small_lp_stft}.

The smoothness of $V_\ph f_a^+$ and $V_\ph f_a^-$ also implies that these functions can be sparsely approximated in space-scale systems, such as two-dimensional wavelets, curvelets, shearlets, etc., given that the frame consists of functions that have fast spatial decay and sufficient regularity. As an example, we refer to \cite[Theorem 8.2]{candes2004new} for the decay of curvelet coefficients of functions that lie in a Sobolev space.

\section{Proof of Proposition \ref{prop:upper-bound}}
\label{sec:proofs}

We start with some simple observations. First, we note that computing $V_\ph \ph$ gives

\begin{equation}\label{Vphi_phi}
V_\ph \ph(x,y) = \frac{1}{\sqrt{2}} e^{-\pi i x y} e^{-\pi/2(x^2+y^2)}.
\end{equation}

A straightforward calculation yields

\begin{equation}\label{VgTa}
	V_\ph T_a f(x,y) = e^{-2 \pi i a y} V_\ph f(x-a,y),
\end{equation}
where $T_a$ is the translation operator defined in (\ref{eqn:translation}). Furthermore, Equations (\ref{Vphi_phi}) and (\ref{VgTa}) imply the following:
\begin{lemma}\label{lem:V_u} Let $u_a$ be the shifted Gaussian defined in (\ref{eqn:gauss_shift}). Then,
$$
	V_\ph u_a (x,y)= \frac{1}{\sqrt{2}} e^{- \pi i a y} e^{-\pi i x y} e^{-\frac{\pi}{2} (x-a)^2} e^{-\frac{\pi}{2} y^2}.
$$
\end{lemma}

\begin{proof}
\begin{align*}
V_\ph u_a (x,y) &= V_\ph T_a \ph (x,y) = e^{-2 \pi i a y} V_\ph \ph (x-a,y) \\
&= e^{-2 \pi i a y} \frac{1}{\sqrt{2}} e^{-\pi i (x-a) y} e^{-\frac{\pi}{2} ((x-a)^2+y^2)} \\
&= \frac{1}{\sqrt{2}} e^{-\pi i a y} e^{-\pi i x y} e^{-\frac{\pi}{2}(x-a)^2} e^{- \frac{\pi}{2} y^2}.
\end{align*} 
\end{proof}
The goal of this section is to find an upper bound on the expression
\begin{equation}\label{eqn:bdball}
	\left\| |V_\ph f_a^+|-|V_\ph f_a^-|\right\|_{W^{1,2}(\mathbb{R}^2)}
\end{equation}
that is of the order $e^{-a^2 \pi/2}$. We start with a pointwise estimate to obtain a bound for the $L^2$-norm.
\subsection{Pointwise estimates}

\begin{lemma}
Let $f_a^+, f_a^-$ be the functions defined in (\ref{eqn:f_a}). For $(x,y) \in \mathbb{R}^2$, the following two inequalities hold:
\begin{align}
	\left|{|V_\ph f_a^+(x,y)|-|V_\ph f_a^-(x,y)|}\right| & \leq \sqrt{2} e^{- \frac{\pi}{2}((x-a)^2+y^2)}, \label{eqn:pw1} \\
	\left|{|V_\ph f_a^+(x,y)|-|V_\ph f_a^-(x,y)|}\right| &\leq \sqrt{2} e^{- \frac{\pi}{2}((x+a)^2+y^2)}. \label{eqn:pw2}
\end{align}
\end{lemma}

\begin{proof}
We only derive the first estimate, the other being analogous:
\begin{align*}
	&\left|{|V_\ph f_a^+(x,y)|-|V_\ph f_a^-(x,y)|}\right| = \\
	&= \frac{1}{\sqrt{2}} \left| |e^{-\frac{\pi}{2} ((x+a)^2+y^2)} + e^{-2 \pi i a y} e^{-\frac{\pi}{2}((x-a)^2+y^2)}| - |e^{-\frac{\pi}{2} ((x+a)^2+y^2)} - e^{-2 \pi i a y} e^{-\frac{\pi}{2}((x-a)^2+y^2)}|\right| \\
	&= \frac{1}{\sqrt{2}} e^{-\frac{\pi}{2} ((x+a)^2+y^2)} \cdot \left| |1+e^{-2\pi i a y} e^{2 a x \pi}|-|1-e^{-2 \pi i a y} e^{2 a x \pi} |\right| \\
	& \leq \frac{1}{\sqrt{2}} e^{-\frac{\pi}{2} ((x+a)^2+y^2)} \cdot 2 e^{2 a x \pi} \\
	& \leq \sqrt{2} e^{-\frac{\pi}{2} ((x-a)^2+y^2)},
\end{align*}	
where the first inequality holds due to the reverse triangle inequality.
\end{proof}
\subsection{$L^2-$estimates}
The previous two estimates can now be combined to obtain an upper bound on the $L^2$-norm. Within $B_{2a}(0)$, i.e. a disk of radius $2a$ around the origin, we will employ (\ref{eqn:pw1}) on the half disk corresponding to $x < 0$ and (\ref{eqn:pw2}) on the half disk corresponding to $x\geq0$. For bounding the portion outside of $B_{2a}(0)$ either of the estimates (\ref{eqn:pw1}), (\ref{eqn:pw2}) will do.
\begin{lemma}\label{lem:normbd}
Let $f_a^+, f_a^-$ be the functions defined in (\ref{eqn:f_a}). Then,
	$$
	\left\| |V_\ph f_a^+|-|V_\ph f_a^-| \right\|_{L^2(\mathbb{R}^2)} \leq  2 \sqrt{1+2a^2\pi} \, e^{-a^2 \pi/2}.
	$$
\end{lemma}

\begin{proof}
	First, we derive a bound on $B_{2a}(0)$ in which we transform to polar coordinates $x = r \cos \ph, y = r \sin \ph$:
	\begin{align*}
	&\int_{\ann} \Big( |V_\ph f_a^+(x,y)|-|V_\ph f_a^-(x,y)| \Big)^2 d(x,y) \leq \\
	&\leq 2 \int_{\pi/2}^{3\pi/2} \int_{0}^{2 a} e^{- \pi (r^2 - 2 r a \cos \ph + a^2)} r \, dr \, d\ph + 2 \int_{-\pi/2}^{\pi/2} \int_{0}^{2 a} e^{- \pi (r^2 + 2 r a \cos \ph + a^2)} r \, dr \, d\ph \\
	&\leq 2 \, e^{-a^2 \pi} \Big( \int_{\pi/2}^{3\pi/2} \int_{0}^{2 a} e^{2 r a \pi \cos \ph} r \, dr \, d\ph + \int_{-\pi/2}^{\pi/2} \int_{0}^{2 a} e^{- 2 r a \pi \cos \ph} r \, dr \, d\ph \Big) \\
	& \leq 2 \, e^{-a^2 \pi}  \int_0^{2\pi} \frac{r^2}{2} \bigg|_{0}^{2a} d\ph = 8a^2 \pi \, e^{- a^2 \pi}, 
	\end{align*}
	where in the last line we use that $e^{2 r a \pi \cos \ph} \leq 1$ for $\ph \in [\frac{\pi}{2}, \frac{3 \pi}{2}]$ and $e^{- 2 r a \pi \cos \ph} \leq 1$ for $\ph \in [-\frac{\pi}{2}, \frac{\pi}{2}]$.
	
	Next, the norm on $\mathbb{R}^2 \backslash \ann$ can be bounded from above using (\ref{eqn:pw1}):
	\begin{align*}
	&\int_{\mathbb{R}^2\backslash \ann}  \Big( |V_\ph f_a^+(x,y)|-|V_\ph f_a^-(x,y)| \Big)^2 d(x,y) \leq \\
	& \leq 2 \int_{\mathbb{R}^2\backslash \ann} e^{-\pi ((x-a)^2+y^2)} d(x,y) \\
	& \leq 2 \int_0^{2 \pi} \int_{2a}^{\infty} e^{-\pi (r^2 + a^2 - 2 a r \cos \ph)} r dr d\ph \\
	& \leq 4 \pi \int_{2a}^{\infty} e^{-\pi (r-a)^2} r dr \leq 4 \pi \int_{2a}^{\infty} e^{-\pi (r-a)^2} 2 (r-a) dr \leq 4 e^{-a^2 \pi},
	\end{align*}
	where the second to last inequality is obtained by using that the integration bounds guarantee $r \leq 2(r-a)$.
\end{proof}

\subsection{$W^{1,2}-$estimates}
What remains to be shown is an upper bound on 
$$
	\|\nabla |V_\ph f_a^+|-\nabla |V_\ph f_a^-|\|_{L^2(\mathbb{R}^2)} = \Big( \|\frac{\partial}{\partial x} |V_\ph f_a^+|-\frac{\partial}{\partial x} |V_\ph f_a^-|\|_{L^2(\mathbb{R}^2)}^2 + \|\frac{\partial}{\partial y} |V_\ph f_a^+|-\frac{\partial}{\partial y} |V_\ph f_a^-|\|_{L^2(\mathbb{R}^2)}^2\Big)^{1/2}.
$$ 
For this, we start with a pointwise estimate on
$$
	\left| \dx |V_\ph f_a^+(x,y)| - \dx |V_\ph f_a^-(x,y) | \right|,
$$
where again we derive two upper bounds between which one needs to switch on $B_{2a}(0)$ depending on the sign of $x$.
\begin{lemma}\label{lem:xneg-dx-norm}
The following bounds hold for the Gabor transform magnitudes of $f_a^+$ and $f_a^-$ for all points $(x,y) \in \mathbb{R}^2$:
\begin{align}
	\left| \dx |V_\ph f_a^+(x,y)| - \dx |V_\ph f_a^-(x,y) | \right| &\leq \sqrt{2} (3 a - x) \pi e^{-\frac{\pi}{2} ((x-a)^2+y^2)}, \label{eqn:pw1-dx}\\
	\left| \dx |V_\ph f_a^+(x,y)| - \dx |V_\ph f_a^-(x,y) | \right| &\leq \sqrt{2} (3 a + x) \pi e^{-\frac{\pi}{2} ((x+a)^2+y^2)}. \label{eqn:pw2-dx}
\end{align}
\end{lemma}
\begin{proof}
To show (\ref{eqn:pw1-dx}), define 
	$$
		\eta_a^+(x,y):= |1+e^{-2 \pi i a y + 2 \pi a x}|
	$$
	so that we can write
	$$
		|V_\ph f_a^+(x,y)| = \frac{1}{\sqrt{2}} e^{-\frac{\pi}{2}((x+a)^2+y^2)} \eta_a^+(x,y)
	$$
	and
	$$
		\dx |V_\ph f_a^+(x,y)| = \frac{1}{\sqrt{2}}(-\pi(x+a)) e^{-\frac{\pi}{2}((x+a)^2+y^2)} \eta_a^+(x,y) +  \frac{1}{\sqrt{2}} e^{-\frac{\pi}{2}((x+a)^2+y^2)} \dx \eta_a^+(x,y).
	$$
	For $\dx \eta_a^+(x,y)$ we obtain
	$$
		\dx \eta_a^+(x,y) = 2 \pi a e^{2 \pi a x} \frac{\cos(2 \pi a y) +e^{2 \pi a x}}{|1+e^{-2 \pi i a y+2 \pi a x}|}.
	$$
	This implies 
	\begin{align}
		\dx |V_\ph f_a^+(x,y)| &= - \frac{1}{\sqrt{2}} \pi (x+a) e^{-\frac{\pi}{2}((x+a)^2+y^2)} |1+e^{-2 \pi i a y+2 \pi a x}| +  \label{eqn:V-dx}\\
		& + \frac{1}{\sqrt{2}} e^{-\frac{\pi}{2}((x+a)^2+y^2)} 2 \pi a e^{2 \pi a x} \frac{\cos(2 \pi a y) +e^{2 \pi a x}}{|1+e^{-2 \pi i a y+2 \pi a x}|}. \nonumber
	\end{align}
	A similar calculation for $\dx |V_\ph f_a^-(x,y)|$ yields
	\begin{align*}
		\dx |V_\ph f_a^-(x,y)| &= - \frac{1}{\sqrt{2}} \pi (x+a) e^{-\frac{\pi}{2}((x+a)^2+y^2)} |1-e^{-2 \pi i a y+2 \pi a x}| +  \\
		& + \frac{1}{\sqrt{2}} e^{-\frac{\pi}{2}((x+a)^2+y^2)} 2 \pi a e^{2 \pi a x} \frac{-\cos(2 \pi a y) +e^{2 \pi a x}}{|1-e^{-2 \pi i a y+2 \pi a x}|}.
	\end{align*}
	We can use these estimates to obtain an upper bound on the distance between the two quantities as follows:
	\begin{align}
		& \left| \dx |V_\ph f_a^+(x,y)| - \dx |V_\ph f_a^-(x,y)| \right| \leq \nonumber \\
		& \leq \frac{1}{\sqrt{2}} |x+a| \pi e^{-\frac{\pi}{2}((x+a)^2+y^2)} \left| |1+e^{-2 \pi i a y+2 \pi a x}| - |1-e^{-2 \pi i a y+2 \pi a x}|\right| + \nonumber  \\
		& + \frac{1}{\sqrt{2}} e^{-\frac{\pi}{2}((x-a)^2+y^2)} 2 \pi a \Big( \frac{|\cos(2 \pi a y) + e^{2 \pi a x}|}{|1+ e^{-2 \pi i a y+ 2 \pi a x}|} + \frac{|\ e^{2 \pi a x}- \cos(2 \pi a y) |}{|1- e^{-2 \pi i a y+ 2 \pi a x}|}\Big) \nonumber  \\
		& \leq \sqrt{2} |x+a| \pi e^{-\frac{\pi}{2}((x+a)^2+y^2)} e^{2 \pi a x} + \sqrt{2} e^{-\frac{\pi}{2}((x-a)^2+y^2)} 2 \pi a \\
		& \leq \sqrt{2} (3 a -x) \pi e^{-\frac{\pi}{2}((x-a)^2+y^2)}, \label{dxest}
	\end{align}
	where in the last step we have used the reverse triangle inequality and
	\begin{align}
	\frac{|\cos(2 \pi a y) + e^{2 \pi a x}|}{|1+ e^{-2 \pi i a y+ 2 \pi a x}|} + \frac{|e^{2 \pi a x} - \cos(2 \pi a y) |}{|1- e^{-2 \pi i a y+ 2 \pi a x}|} \leq 2.\label{eqn:bound-2nd-term}
	\end{align}
	Similarly, Equation (\ref{eqn:pw2-dx}) can be shown by defining
$$
		\widetilde{\eta}_a^+(x,y):= |e^{-2 \pi a x}+ e^{-2 \pi i a y} |
	$$
	which allows to write
	$$
		|V_\ph f_a^+(x,y)| = \frac{1}{\sqrt{2}} e^{-\frac{\pi}{2}((x-a)^2+y^2)} \widetilde{\eta}_a^+(x,y).
	$$
\end{proof}

Combining the inequalities in (\ref{eqn:pw1-dx}) and (\ref{eqn:pw2-dx}), we arrive at:
\begin{lemma}\label{lem:normbd-dx} There exists a uniform constant $C_1>0$ such that for all $a>0$,
$$
\Big\| \dx |V_\ph f_a^+|-\dx |V_\ph f_a^-|  \Big\|_{L^2(\mathbb{R}^2)} \leq C_1 \,a^2 \, e^{-a^2 \pi/2}.
$$
\end{lemma}
\begin{proof}
As before, the norm on $B_{2a}(0)$ can be bounded by employing (\ref{eqn:pw1-dx}) and (\ref{eqn:pw2-dx}) for the cases $x<0$ and $x \geq 0$, respectively:
\begin{align*}
	&\int_{\ann} \Big( \dx |V_\ph f_a^+(x,y)|- \dx |V_\ph f_a^-(x,y)| \Big)^2 d(x,y) \leq \\
	&\leq 2 \pi^2 \int_{\pi/2}^{3\pi/2} \int_{0}^{2 a} (3 a - r \cos \ph)^2 e^{- \pi (r^2 - 2 r a \cos \ph + a^2)} r \, dr \, d\ph + \\
	& + 2 \pi^2 \int_{-\pi/2}^{\pi/2} \int_{0}^{2 a} (3 a + r \cos \ph)^2 e^{- \pi (r^2 + 2 r a \cos \ph + a^2)} r \, dr \, d\ph \\
	&\leq 50 a^2 \pi^2 \, e^{-a^2 \pi} \Big( \int_{\pi/2}^{3\pi/2} \int_{0}^{2 a} e^{2 r a \pi \cos \ph} r \, dr \, d\ph + \int_{-\pi/2}^{\pi/2} \int_{0}^{2 a} e^{- 2 r a \pi \cos \ph} r \, dr \, d\ph \Big) \\
	& \leq 200 a^4 \pi^3 \, e^{-a^2 \pi},
	\end{align*}
	where the last line is obtained from noting that for the first summand, $\cos \ph \leq 0$, so that $e^{2 r a \pi \cos \ph} \leq 1$ and for the second summand, $\cos \ph \geq 0$, so that $e^{-2 r a \pi \cos \ph} \leq 1$. Thus, each of the two 		summands can be bounded by $2 a^2 \pi$. 
	For the estimation of the norm on $\mathbb{R}^2 \backslash \ann$, one can use the pointwise estimate (\ref{eqn:pw1-dx}) as follows:
	\begin{align*}
	&\int_{\mathbb{R}^2 \backslash \ann} \Big( \dx |V_\ph f_a^+(x,y)|- \dx |V_\ph f_a^-(x,y)| \Big)^2 d(x,y) \leq \\
	& \leq 2 \pi^2 \int_{\mathbb{R}^2 \backslash \ann} (3 a -x)^2 e^{-\pi((x-a)^2+y^2)} d(x,y) \\
	& \leq 2 \pi^2 \int_{0}^{2\pi} \int_{2a}^{\infty} (3 a -r \, \cos \ph)^2 e^{-\pi(r^2-2 a r \cos \ph + a^2)} r dr d\ph \\
	& \leq 2 \pi^2  \int_{0}^{2\pi} \int_{2a}^{\infty} \frac{25}{4} r^3 e^{-\pi(r-a)^2} dr d\ph \\
	& \leq 25 \pi^3 \int_{2a}^{\infty} r^3 e^{-\pi(r-a)^2} dr.
	\end{align*}
	Simplification of the last integral, yields
	\begin{align*}
	&25 \pi^3 \int_{2a}^{\infty} r^3 e^{-\pi(r-a)^2} dr \leq \\
	&\leq \frac{25}{4} \pi \Big(2 e^{-a^2 \pi} (1+7 a^2 \pi) + a \pi (3+2 a^2 \pi) \int_{a \sqrt{\pi}}^{\infty} e^{-t^2} dt \Big)\\
	& \leq \frac{25}{4} \pi \Big(2 e^{-a^2 \pi} (1+7 a^2 \pi) + a \pi (3+2 a^2 \pi) \int_{a \sqrt{\pi}}^{\infty} \frac{t}{a \sqrt{\pi}} e^{-t^2} dt \Big) \\
	& \leq \frac{25}{4} \pi \big( (2+\frac{3}{2} \sqrt{\pi}) + a^2 \pi (14 + \sqrt{\pi})\big) e^{-a^2 \pi}.
	\end{align*}
\end{proof}
We conclude this section by providing the last missing upper bound, the one on the $L^2$-norm of the partial derivatives w.r.t. $y$. Again, we can find pointwise estimates:
\begin{lemma}\label{lem:xneg-dy-norm}
The following bounds hold for the Gabor transform magnitudes of $f_a^+$ and $f_a^-$ for all points $(x,y) \in \mathbb{R}^2$:
\begin{align}
	\left| \dy |V_\ph f_a^+(x,y)| - \dy |V_\ph f_a^-(x,y)| \right| & \leq \sqrt{2} (|\pi y| + 2 \pi a) e^{-\frac{\pi}{2} ((x-a)^2+y^2) }, \label{eqn:pw1-dy} \\
	\left| \dy |V_\ph f_a^+(x,y)| - \dy |V_\ph f_a^-(x,y)| \right| & \leq \sqrt{2} (|\pi y| + 2 \pi a) e^{-\frac{\pi}{2} ((x+a)^2+y^2) } . \label{eqn:pw2-dy}
\end{align}
\end{lemma}
\begin{proof}
To show the validity of (\ref{eqn:pw1-dy}), we proceed as in the proof of Lemma \ref{lem:xneg-dx-norm} and obtain
\begin{align}
	\dy |V_\ph f_a^+(x,y)| &= \frac{1}{\sqrt{2}} (-\pi y) e^{-\frac{\pi}{2}((x+a)^2 + y^2)} |1+e^{-2 \pi i a y + 2 \pi a x}| \nonumber \\
	& +  \frac{1}{\sqrt{2}}e^{-\frac{\pi}{2}((x+a)^2 + y^2)} e^{2 \pi a x} 2 \pi a  \frac{-  \sin(2 \pi a y)}{ |1+e^{-2 \pi i a y + 2 \pi a x}|}\label{eqn:V-dy},
\end{align}
as well as
\begin{align*}
	\dy |V_\ph f_a^-(x,y)| &= \frac{1}{\sqrt{2}} (-\pi y) e^{-\frac{\pi}{2}((x+a)^2 + y^2)} |1-e^{-2 \pi i a y + 2 \pi a x}| \\
	& +  \frac{1}{\sqrt{2}}e^{-\frac{\pi}{2}((x+a)^2 + y^2)} e^{2 \pi a x} 2 \pi a \frac{   \sin(2 \pi a y)}{ |1-e^{-2 \pi i a y + 2 \pi a x}|}.
\end{align*}
Noting that
\begin{equation}\label{bound-2nd-term-y}
	\left| \frac{ \sin(2 \pi a y)}{ |1+e^{-2 \pi i a y + 2 \pi a x}|} \right| \leq 1, \qquad \left| \frac{\sin(2 \pi a y)}{ |1-e^{-2 \pi i a y + 2 \pi a x}|} \right| \leq 1
\end{equation}
and repeating the arguments of the proof of Lemma \ref{lem:xneg-dx-norm}, one arrives at
\begin{align*}
	& \left| \dy |V_\ph f_a^+(x,y)| - \dy |V_\ph f_a^-(x,y)| \right| \\
	& \leq \frac{1}{\sqrt{2}} |\pi y| e^{-\frac{\pi}{2}((x+a)^2+y^2)} \left| |1+e^{-2 \pi i a y+2 \pi a x}|-|1-e^{-2 \pi i a y+2 \pi a x}|\right| + \\
	& + \frac{1}{\sqrt{2}} e^{-\frac{\pi}{2}((x-a)^2+y^2)} 2 \pi a \left|  \frac{ \sin(2 \pi a y)}{ |1+e^{-2 \pi i a y + 2 \pi a x}|} +  \frac{   \sin(2 \pi a y)}{ |1-e^{-2 \pi i a y + 2 \pi a x}|}\right| \\
	& \leq  \sqrt{2} e^{-\frac{\pi}{2}((x-a)^2+y^2)} (|\pi y| + 2 \pi a).
\end{align*}
The proof of (\ref{eqn:pw2-dy}) is similar.
\end{proof}
As in the case of the partial derivatives with respect to $x$, one can derive an estimate as in Lemma \ref{lem:normbd-dx}:
\begin{lemma}\label{lem:normbd-dy}
There exists a uniform constant $C_2>0$ such that for all $a>0$ the following holds:
	$$
\Big\| \dy |V_\ph f_a^+|-\dy |V_\ph f_a^-|  \Big\|_{L^2(\mathbb{R}^2)} \leq C_2 \, a^2 \, e^{-a^2 \pi/2}.
$$
\end{lemma}

\section{Proofs of Section \ref{sec:sparse_reg}}\label{sec:proofs_sparse_reg}

\begin{proof}[Proof of Proposition \ref{prop:small_lp_stft}]
Since $p \in [1,2]$ and we can assume $\|g\|_{L^2(\mathbb{R})}\leq1$ without loss of generality, the basic estimate 
$$
\Big||\langle f_a^+, g_{n,k} \rangle|^p - |\langle f_a^-, g_{n,k} \rangle |^p \Big| \leq 2 \,\Big||\langle f_a^+, g_{n,k} \rangle| - |\langle f_a^-, g_{n,k} \rangle | \Big|,
$$
holds. Applying the reverse triangle inequality twice on the RHS of the above, one obtains (where we recall the definition of $u_a, u_{-a}$ in (\ref{eqn:gauss_shift}):
$$
\Big||\langle f_a^+, g_{n,k} \rangle|^p - |\langle f_a^-, g_{n,k} \rangle |^p \Big| \leq 4 \min\{ | \langle u_{-a}, g_{n,k}\rangle|, |\langle u_{a}, g_{n,k} \rangle|\}.
$$
Thus,
\begin{align}
\sum_{n,k \in \mathbb{Z}} & \Big( |\langle f_a^+, g_{n,k} \rangle|^p - |\langle f_a^-, g_{n,k} \rangle |^p \Big) w(n x_0, k y_0)^p \leq \nonumber \\
& \leq 4 \sum_{\substack{n \in \mathbb{Z}_0^+ \\ k \in \mathbb{Z}}} | \langle u_{-a}, g_{n,k}\rangle| w(n x_0, k y_0)^p + 4 \sum_{\substack{n \in \mathbb{Z}^- \\ k \in \mathbb{Z}}} | \langle u_{a}, g_{n,k}\rangle| w(n x_0, k y_0)^p. \label{eqn:lp-diff}
\end{align}
The property that $V_g \ph \in \mathcal{S}(\mathbb{R}^2)$, implies that for any $m \in \mathbb{N}$ there exists a constant $\widetilde{C}_m>0$, s.t.
$$
\big| V_g \ph (x,y) \big| \leq \widetilde{C}_m \frac{1}{\big((1+|x|)(1+|y|)\big)^m}.
$$
Further note that $| V_g u_a(x,y)| = |V_g \ph (x-a,y)|$, $| V_g u_{-a} (x,y)| = |V_g \ph (x+a,y)|$ and hence, 
\begin{align*}
|\langle u_{-a}, g_{n,k} \rangle| w(n x_0, k y_0)^p \leq \widetilde{C}_m \frac{\big((1+|x_0n|)(1+|y_0k|)\big)^{sp}}{\big((1+|x_0n+a|)(1+|y_0k|)\big)^m} \leq \widetilde{C}_m  \big((1+|x_0 n+a|)(1 + |y_0 k|)\big)^{sp-m},
\end{align*}
for $n \geq 0$, and similarly
$$
|\langle u_{a}, g_{n,k} \rangle| w(n x_0, k y_0)^p \leq \widetilde{C}_m  \big((1+|x_0 n-a|)(1 + |y_0 k|)\big)^{sp-m}, \quad \mbox{for } n<0.
$$
These estimates can now be inserted into (\ref{eqn:lp-diff}) to obtain
\begin{align*}
\sum_{n,k \in \mathbb{Z}} & \Big( |\langle f_a^+, g_{n,k} \rangle|^p - |\langle f_a^-, g_{n,k} \rangle |^p \Big) w(n x_0, k y_0)^p \leq \\
& \leq 4 \widetilde{C}_m  \Big( \sum_{\substack{n \in \mathbb{Z}_0^+ \\ k \in \mathbb{Z}}}  \big((1+|n x_0 +a|)(1 + |k y_0 |)\big)^{sp-m} + \sum_{\substack{n \in \mathbb{Z}^- \\ k \in \mathbb{Z}}}  \big((1+|n x_0 -a|)(1 + |k y_0 |)\big)^{sp-m} \Big) \\
& \leq 4 \widetilde{C}_m \Big( (1+a)^{s p-m} + 2 \int_0^\infty (1+x x_0 + a)^{sp-m} dx \Big) \cdot \Big( 1+ 2 \int_0^\infty (1+y y_0)^{sp-m} dy \Big) \\
& \leq C_m \big(1 + a\big)^{s p - m +1},
\end{align*}
for some constant $C_m$ depending on $m, x_0$ and $y_0$, where the last line follows from $m>sp+1$.
\end{proof}

\begin{proof}[Proof of Proposition \ref{prop:wavelet-decay}]
Let $p_{w,m}(t)$ be the Taylor polynomial of $\ph$ at point $w := \alpha^{-j} \beta k$ and of degree $m-1$. Then, 
\begin{equation}\label{eqn:w-nbhd}
|\ph(t) - p_{w,m}(t)| \leq e^{-\pi w^2/8} |t-w|^m, \quad \mbox{ for all } t \in I_w:=[w/2, 3w/2].
\end{equation}
On the other hand, since $\psi$ has $m$ vanishing moments, we have that
$$
\langle \ph, \psi_{j,k} \rangle = \langle \ph - p_{w,m}, \psi_{j,k} \rangle = \int_{\mathbb{R}} (\ph(t) - p_{w,m}(t)) \alpha^{j/2} \psi(\alpha^j t - \beta k) dt.
$$
To estimate this integral, we split it into a piece on $I_w$ and a piece on the remaining interval on which $\psi_{j,k}$ has small energy. Substituting $x = \alpha^j t - \beta k$ we obtain
\begin{align*}
|\langle \ph, \psi_{j,k} \rangle | & \leq  | \int_{I_w} (\ph(t) - p_{\alpha^{-j} \beta k,m}(t)) \alpha^{j/2} \psi(\alpha^j t - \beta k) dt|  + | \int_{\mathbb{R}\backslash I_w} (\ph(t) - p_{\alpha^{-j}\beta k,m}(t)) \alpha^{j/2} \psi(\alpha^j t - \beta k) dt| \\
&\leq e^{-\pi w^2/8} | \int_{I_w} |t-\alpha^{-j} \beta k|^{m} \alpha^{j/2} \psi(\alpha^j t - \beta k) dt| + | \int_{\mathbb{R}\backslash I_w} |t-\alpha^{-j} \beta k|^{m} \alpha^{j/2} \psi(\alpha^j t - \beta k) dt| \\
& \leq e^{-\pi w^2/8} \alpha^{-j(m+3/2)}  \int_{-\frac{\beta |k|}{2}}^{\frac{\beta |k|}{2}} |x|^{m} |\psi(x)| dx + \alpha^{-j(m+3/2)} \int_{\mathbb{R}\backslash [-\frac{\beta |k|}{2},\frac{\beta |k|}{2}]} |x|^{m} |\psi(x)| dx.
\end{align*}
Employing the decay property (\ref{eqn:wavelet-decay}) of $\psi$ further yields
\begin{align*}
|\langle \ph, \psi_{j,k} \rangle | & \leq \alpha^{-j(m+3/2)}  C_{2m+2} \Big(e^{-\pi w^2/8}  \int_{-\frac{\beta |k|}{2}}^{\frac{\beta |k|}{2}} \frac{|x|^{m}}{1+|x|^{2m+2}} dx + \int_{\mathbb{R}\backslash [-\frac{\beta |k|}{2},\frac{\beta |k|}{2}]} \frac{|x|^{m}}{1+|x|^{2m+2}} dx \Big) \\
& \leq \alpha^{-j(m+3/2)}  C_{2m+2} \Big(e^{-\pi w^2/8}  \beta |k| + \int_{\mathbb{R}\backslash [-\frac{\beta |k|}{2},\frac{\beta |k|}{2}]}  |x|^{-m-2} dx \Big) \\
& \leq \alpha^{-j(m+3/2)}  C_{2m+2} \Big( \beta |k| \, e^{-\pi (\alpha^{-j} \beta k)^2/8} + \frac{2^{m+2}}{(\beta |k|)^{m+1}} \Big) \\
& \leq C \, \alpha^{-j(m+3/2)} (\beta |k|)^{-m-1}.
\end{align*}
To show inequality (\ref{eqn:scaling-coeff-decay}), we note that for $\beta |k| \geq 2$:
\begin{align*}
|\langle \ph, \chi_{0,k} \rangle| &\leq \widetilde{C}_{m+2} \int_{\mathbb{R}} e^{-\pi t^2} \frac{1}{1+|t-\beta k|^{m+2}} dt \\
& \leq \widetilde{C}_{m+2} \Big( \int_{-\frac{\beta |k|}{2}}^{\frac{\beta |k|}{2}} \frac{1}{1+|t-\beta k|^{m+2}} dt + \int_{\mathbb{R} \backslash [-\frac{\beta |k|}{2}, \frac{\beta |k|}{2}]} e^{-\pi t^2} dt \Big) \\
& \leq \widetilde{C}_{m+2} \big( 2^{m+2} (\beta |k|)^{-m-1} + \frac{1}{\pi} e^{-\pi (\beta k)^2/4} \big) \leq \widetilde{C} (\beta |k|)^{-m-1}.
\end{align*}
The constant $C:= \max\{C,\widetilde{C}\}$ depends on only $m$.
\end{proof}
\textbf{Acknowledgments.} PG gratefully acknowledges the support by the Austrian Science Fund (FWF) under grant P 30148.

\bibliographystyle{abbrv}
\bibliography{Deep_Learning.bib}

\begin{thebibliography}{10}

\bibitem{alaifari2016stable}
R.~Alaifari, I.~Daubechies, P.~Grohs, and R.~Yin.
\newblock Stable phase retrieval in infinite dimensions.
\newblock {\em Foundations of Computational Mathematics, to appear. Preprint
  available from arXiv:1609.00034}, 2018.

\bibitem{alaifari2017phase}
R.~Alaifari and P.~Grohs.
\newblock Phase retrieval in the general setting of continuous frames for
  {Banach} spaces.
\newblock {\em SIAM Journal on Mathematical Analysis}, 49(3):1895--1911, 2017.

\bibitem{balan2006signal}
R.~Balan, P.~Casazza, and D.~Edidin.
\newblock On signal reconstruction without phase.
\newblock {\em Applied and Computational Harmonic Analysis}, 20(3):345--356,
  2006.

\bibitem{bandeira2014saving}
A.~S. Bandeira, J.~Cahill, D.~G. Mixon, and A.~A. Nelson.
\newblock Saving phase: Injectivity and stability for phase retrieval.
\newblock {\em Applied and Computational Harmonic Analysis}, 37(1):106--125,
  2014.

\bibitem{cahill2016phase}
J.~Cahill, P.~Casazza, and I.~Daubechies.
\newblock {Phase retrieval in infinite-dimensional Hilbert spaces}.
\newblock {\em Transactions of the American Mathematical Society, Series B},
  3(3):63--76, 2016.

\bibitem{candes2004new}
E.~J. Cand{\`e}s and D.~L. Donoho.
\newblock New tight frames of curvelets and optimal representations of objects
  with piecewise {$C^2$} singularities.
\newblock {\em Communications on Pure and Applied Mathematics}, 57(2):219--266,
  2004.

\bibitem{daubechies2004iterative}
I.~Daubechies, M.~Defrise, and C.~De~Mol.
\newblock An iterative thresholding algorithm for linear inverse problems with
  a sparsity constraint.
\newblock {\em Communications on Pure and Applied Mathematics},
  57(11):1413--1457, 2004.

\bibitem{engl1996regularization}
H.~W. Engl, M.~Hanke, and A.~Neubauer.
\newblock {\em Regularization of inverse problems}.
\newblock Springer Science \& Business Media, 1996.

\bibitem{grochenig}
K.~Gr{\"o}chenig.
\newblock {\em Foundations of time-frequency analysis}.
\newblock Springer Science \& Business Media, 2013.

\bibitem{grohs2017stable}
P.~Grohs and M.~Rathmair.
\newblock {Stable Gabor Phase retrieval and Spectral Clustering}.
\newblock {\em Communications on Pure and Applied Mathematics, to appear.
  Preprint available from arXiv:1706.04374.}, 2018.

\bibitem{mallat1999wavelet}
S.~Mallat.
\newblock {\em A wavelet tour of signal processing}.
\newblock Academic press, 1999.

\bibitem{mallat2015phase}
S.~Mallat and I.~Waldspurger.
\newblock Phase retrieval for the cauchy wavelet transform.
\newblock {\em Journal of Fourier Analysis and Applications}, 21(6):1251--1309,
  2015.

\end{thebibliography}

\end{document}